\documentclass[abstract=on]{scrartcl}

 \usepackage{amsmath,amsthm,amsfonts,amssymb}
 \usepackage{hyperref} 
 \usepackage{mathrsfs} 

 \usepackage{a4wide}
 \usepackage{mathptmx}

  \newtheorem{theorem}{Theorem}[section]
\newtheorem{corollary}[theorem]{Corollary}

\newtheorem{lemma}[theorem]{Lemma}
\newtheorem{proposition}[theorem]{Proposition}

\theoremstyle{definition}
\newtheorem{definition}[theorem]{Definition}
\newtheorem{remark}{Remark}

\begin{document}
 
\title{Multiple $\mathbb{S}^{1}$-orbits for  \\ the Schr\"{o}dinger-Newton system }
\author{Silvia Cingolani\thanks{Supported by the MIUR project \textit{Variational and
topological methods in the study of nonlinear phenomena} 
   (PRIN 2009) and by GNAMPA project 2012.}\\
  Dipartimento di Meccanica, Matematica e Management \\ Politecnico di Bari
  \\
   via Orabona 4, 70125 Bari, Italy \and
   Simone Secchi\thanks{Supported by the PRIN 
   2009 project \textit{Teoria dei punti critici
e metodi perturbativi per equazioni differenziali nonlineari}.} \\
   Dipartimento di Matematica e Applicazioni\\  Universit\`a di Milano-Bicocca
   \\
    via R. Cozzi 53, I-20125 Milano, Italy
   }
   \date{}
\maketitle    

\begin{abstract}
We prove existence and multiplicity of symmetric solutions for the
\emph{Schr\"{o}dinger-Newton system} in three
 dimensional space using equivariant Morse theory.
\end{abstract}


\section{Introduction}

The \emph{Schr\"{o}dinger-Newton system} in three dimensional space
   has a long standing history.
It was firstly proposed in 1954 by Pekar for describing the
    quantum mechanics of a polaron. Successively it was derived by
      Choquard  for describing an electron trapped in its own
       hole and  by  Penrose  \cite{pe1, pe2, pe3}
in his discussions on the selfgravitating matter.

For a single particle of mass $m$ the system is
obtained by coupling together the linear Schr\"{o}dinger equation
of quantum mechanics with the Poisson equation from 
       Newtonian mechanics.  It has the form
\begin{equation}
\begin{cases}
-\frac{\hbar^{2}}{2m}\Delta\psi+V(x)\psi+U\psi=0,\\
-\Delta U+4\pi\kappa|\psi|^{2}=0,
\end{cases}
\label{sys:s-n}%
\end{equation}
where $\psi$ is the complex wave function, $U$ is the gravitational potential
energy, $V$ is a given potential, $\hbar$ is Planck's constant, and
$\kappa:=\mathrm{G}m^{2}$, $\mathrm{G}$ being Newton's constant.

Rescaling
$\psi(x)=\frac{1}{\hbar}\frac{\hat{\psi}(x)}{\sqrt{2\kappa m}}$,
$V(x)=\frac{1}{2m}\hat{V}(x)$, $U(x)=\frac{1}{2m}\hat{U}(x),$
system~\eqref{sys:s-n} becomes equivalent to the single nonlocal
equation
\begin{equation}
-\hbar^{2}\Delta\hat{\psi}+\hat{V}(x)\hat{\psi}=\frac{1}{\hbar^{2}}
 \Big (  \frac{1}{\left\vert x\right\vert }\ast|\hat{\psi}|^{2} \Big )  \hat{\psi}.\\
 \label{eq:1.3}
\end{equation}

The existence of one solution can be traced back to Lions' paper~\cite{l}.
Successively  equation $(\ref{eq:1.3})$ and related equations
have been investigated by many authors, see e.g. \cite{a, fl, lieb,
fty, hmt, gv, mz, mt,  mpt, n, secchi, t, ccs2, ms} and the references therein.
Semiclassical analysis for equation (\ref{eq:1.3}) has been studied in \cite{ww} and in \cite{css}
for a more general convolution potential, not necessarily radially symmetric.

In this work we shall consider the nonlocal equation $(\ref{eq:1.3}$)
in presence of a  magnetic potential $A$ and an electric potential $V$ which satisfy
specific symmetry. Precisely, we consider $G$ a closed subgroup of the group 
       $O(3)$ of linear isometries of $\mathbb{R}^{3}$ and assume that 
       $A\colon \mathbb{R}^{3}\rightarrow\mathbb{R}^{3}$ is a $C^{1}%
$-function, and $V\colon \mathbb{R}^{3}\rightarrow\mathbb{R}$ is a bounded
continuous function with $\inf_{\mathbb{R}^{3}}V>0$, which satisfy
\begin{equation}
A(gx)=gA(x)\text{ \ \ and \ \ }V(gx)=V(x)\quad\text{for all $g\in G$,
$x\in\mathbb{R}^{3}$}. \label{G}
\end{equation}
Given a continuous homomorphism of groups $\tau\colon G\rightarrow\mathbb{S}^{1}$
into the group $\mathbb{S}^{1}$ of unit complex numbers. A physically relevant 
       example is a constant magnetic field $B= \operatorname{curl}A = (0,0,2)$ 
       and the group $G_m=\{e^{2\pi i k /m} \mid k=1,\ldots,m\}$ for 
       $m \in \mathbb{N}$, $m \geq 1$; see Subsection \ref{subs:5.1} for more details.

We are interested in semiclassical states, i.e.
solutions as $\varepsilon\rightarrow0$
to the problem
\begin{equation}
\begin{cases}
\left(  -\varepsilon\mathrm{i}\nabla+A\right)^{2}u+V(x)u=\frac{1}{\varepsilon^{2}}
\left(  \frac{1}{\left\vert x\right\vert }\ast|u|^{2}\right)  u,\\
u\in L^{2}(\mathbb{R}^{3},\mathbb{C}),\\
\varepsilon\nabla u+\mathrm{i}Au\in L^{2}(\mathbb{R}^{3},\mathbb{C}^{3}),
\end{cases}
\label{prob}%
\end{equation}
which satisfy
\begin{equation}
u(gx)=\tau(g)u(x)\text{ \ \ \ for all }g\in G,\text{ }x\in\mathbb{R}^{3}.
\label{tau-inv}%
\end{equation}
This implies that the absolute value $\left\vert u\right\vert $ of $u$ is
$G$-invariant and
 the phase of $u(gx)$ is that of $u(x)$ multiplied by $\tau(g)$.

Recently in \cite{ccs} the authors have showed that there is a combined effect 
       of the symmetries and the electric potential $V$ on the number of 
       semiclassical $\tau$-intertwining solutions to
$(\ref{prob}).$ More precisely, we showed that the Lusternik-Schnirelmann
category of the $G$-orbit space of a suitable set $M_{\tau},$ depending on $V$
and $\tau,$ furnishes a lower bound on the number of solutions of this type.
 In this work we shall apply equivariant Morse theory for
 better multiplicity results than those given by Lusternik-Schnirelmann category.
 Moreover equivariant Morse theory provides information on the local behavior of a
functional around a critical orbit. The main result is established in Theorem \ref{mainthm}.
For the local case, similar results are obtained in \cite{cc2}. 
       For other results about magnetic Schr\"{o}dinger equations, we refer to \cite{CS, CIS}.

Finally, concerning magnetic Pekar functional, we mention the recent results in \cite{ghw}.

\section{The variational problem}

\label{secvarprob}Set $\nabla_{\varepsilon,A}u=\varepsilon\nabla
u+\mathrm{i}Au$ and consider the real Hilbert space%
\[
H_{\varepsilon,A}^{1}(\mathbb{R}^{3},\mathbb{C}):=\{u\in L^{2}(\mathbb{R}%
^{3},\mathbb{C})\mid \nabla_{\varepsilon,A}u\in L^{2}(\mathbb{R}^{3}%
,\mathbb{C}^{3})\}
\]
with the scalar product
\begin{equation}
\left\langle u,v\right\rangle _{\varepsilon,A,V}=\operatorname{Re}%
\int_{\mathbb{R}^{3}}\left(  \nabla_{\varepsilon,A}u\cdot\overline
{\nabla_{\varepsilon,A}v}+V(x)u\overline{v}\right)  . \label{sp}%
\end{equation}
We write
\[
\left\Vert u\right\Vert _{\varepsilon,A,V}= \Big (  \int_{\mathbb{R}^{3}%
}\left(  \left\vert \nabla_{\varepsilon,A}u\right\vert ^{2}+V(x)\left\vert
u\right\vert ^{2}\right)   \Big )^{1/2}%
\]
for the corresponding norm.

If $u\in H_{\varepsilon,A}^{1}(\mathbb{R}^{3},\mathbb{C}),$ then $\left\vert
u\right\vert \in H^{1}(\mathbb{R}^{3},\mathbb{R})$ and
the diamagnetic inequality \cite{ll} holds
\begin{equation}
\varepsilon\left\vert \nabla|u(x)|\right\vert \leq\left\vert \varepsilon\nabla
u(x)+\mathrm{i}A(x)u(x)\right\vert \text{\quad for a.e. $x\in\mathbb{R}^{3}$}.
 \label{di}%
\end{equation}
 Set
\[
\mathbb{D}(u)=\int_{\mathbb{R}^{3}}\int_{\mathbb{R}^{3}}\frac{|u(x)|^{2}%
|u(y)|^{2}}{\left\vert x-y\right\vert }\,dxdy.
\] 
We need some basic inequalities about convolutions.
        A proof can be found in \cite[Theorem 4.3]{ll} and in \cite{lieb83}.
 
\begin{theorem} \label{thm:HLS}
If $p$, $q \in (1,+\infty)$ satisfy $1/p+1/3=1+1/q$ and $f \in L^p(\mathbb{R}^3)$ then
\begin{equation} \label{eq:22}
\left\||x|*f \right\|_{L^q(\mathbb{R}^3)} \leq N_{p} \|f\|_{L^p(\mathbb{R}^3)}
\end{equation}
for a constant $N_p>0$ that depends on $p$ but not on $f$.
More generally, if $p$, $t \in (1,+\infty)$ satisfy $1/p+1/t+1/3 = 2$ 
        and $f \in L^p(\mathbb{R}^3)$, $g \in L^t (\mathbb{R}^3)$, then
\begin{equation}\label{eq:23}
 \Big | \int_{\mathbb{R}^3 \times \mathbb{R}^3} 
        \frac{f(x)g(y)}{|x-y|}\, dx\, dy  \Big | \leq N_p \|f\|_{L^p(\mathbb{R}^3)} \|g\|_{L^t(\mathbb{R}^3)}
\end{equation}
for some constant $N_p>0$ that does not depend on $f$ and $g$.
\end{theorem}

Theorem \ref{thm:HLS} yields that
\begin{equation}
\mathbb{D}(u)\leq C\Vert u\Vert_{L^{12/5}(\mathbb{R}^{3})}^{4} \label{D}%
\end{equation}
for every $u\in H_{\varepsilon,A}^{1}(\mathbb{R}^{3},\mathbb{C}).$

The energy functional $J_{\varepsilon,A,V}\colon H_{\varepsilon,A}^{1}%
(\mathbb{R}^{3},\mathbb{C})\rightarrow\mathbb{R}$ associated to problem
(\ref{prob}), defined by
\[
J_{\varepsilon,A,V}(u)=\frac{1}{2}\left\Vert u\right\Vert _{\varepsilon
,A,V}^{2}-\frac{1}{4\varepsilon^{2}}\mathbb{D}(u),
\]
is of class $C^{1}$, and its first derivative is given by%
\[
J_{\varepsilon,A,V}^{\prime}(u)[v]=\left\langle u,v\right\rangle _{\varepsilon
,A,V}-\frac{1}{\varepsilon^{2}}\operatorname{Re}\int_{\mathbb{R}^{3}} \Big (
\frac{1}{\left\vert x\right\vert }\ast\left\vert u\right\vert ^{2} \Big )
u\overline{v}.
\]
Moreover  we can write the second derivative
\begin{equation*}
J_{\varepsilon,A,V}''(u)[v,w] = \left\langle w,v\right\rangle_{\varepsilon
,A,V}-  \frac{1}{\varepsilon^{2}}\operatorname{Re}\int_{\mathbb{R}^{3}} \Big (
\frac{1}{\left\vert x\right\vert }\ast\left\vert u\right\vert ^{2} \Big )
w\overline{v} -
 \frac{2}{\varepsilon^{2}}\operatorname{Re}\int_{\mathbb{R}^{3}}
 \Big (
\frac{1}{\left\vert x\right\vert } \ast \left(  u \bar w \right)  \Big ) u \overline{v} .
\end{equation*}
By $(\ref{eq:23})$ it is easy to recognize that
\begin{align*}
|J_{\varepsilon,A,V}''(u)[v,w]| & \leq \| v\|_{\varepsilon, A, V} \ \| w\|_{\varepsilon, A, V} +
C \Vert u \Vert^2_{L^{12/5}(\mathbb{R}^{3})}\Vert v \Vert_{L^{12/5}
        (\mathbb{R}^{3})} \Vert w \Vert_{L^{12/5}(\mathbb{R}^{3})} \\ &
\leq K \| v\|_{\varepsilon, A, V} \ \| w\|_{\varepsilon, A, V}.
\end{align*}
We postpone the proof that $J_{\varepsilon,A,V}$ is of class $C^2$ to the Appendix.

The solutions to problem (\ref{prob}) are the critical points of
$J_{\varepsilon,A,V}$.
The action of $G$ on $H_{\varepsilon,A}^{1}(\mathbb{R}^{3},\mathbb{C})$
defined by $(g,u)\mapsto u_{g},$ where%
\[
(u_{g})(x)=\tau(g)u(g^{-1}x),
\]
satisfies%
\[
\left\langle u_{g},v_{g}\right\rangle _{\varepsilon,A,V}=\left\langle
u,v\right\rangle _{\varepsilon,A,V}\text{ \ \ \ and \ \ \ }\mathbb{D}%
(u_{g})=\mathbb{D}(u)
\]
for all $g\in G,$ $u,v\in H_{\varepsilon,A}^{1}(\mathbb{R}^{3},\mathbb{C}).$
Hence, $J_{\varepsilon,A,V}$ is $G$-invariant. By the principle of symmetric
criticality \cite{p, w}, the critical points of the restriction of
$J_{\varepsilon,A,V}$ to the fixed point space of this $G$-action, denoted by
\begin{align*}
H_{\varepsilon,A}^{1}(\mathbb{R}^{3},\mathbb{C})^{\tau}  &  =\left\{  u\in
H_{\varepsilon,A}^{1}(\mathbb{R}^{3},\mathbb{C})\mid u_{g}=u\right\} \\
&  =\left\{  u\in H_{\varepsilon,A}^{1}(\mathbb{R}^{3},\mathbb{C}%
)\mid u(gx)=\tau(g)u(x)\quad\forall x\in\mathbb{R}^{3},\ g\in
G\right\}  ,
\end{align*}
are the solutions to problem (\ref{prob}) which satisfy (\ref{tau-inv}).

Let us define the \emph{Nehari manifold}%
\[
\mathcal{N}_{\varepsilon,A,V}^{\tau}=\left\{  u\in H_{\varepsilon,A}%
^{1}(\mathbb{R}^{3},\mathbb{C})^{\tau}\mid \text{$u\neq 0$ and $\varepsilon^{2}\left\Vert
u\right\Vert _{\varepsilon,A,V}^{2}=\mathbb{D}(u)$}\right\}  ,
\]
which is a $C^{2}$-manifold radially diffeomorphic to the unit sphere in
$H_{\varepsilon,A}^{1}(\mathbb{R}^{3},$ $\mathbb{C})^{\tau}.$ The critical points
of the restriction of $J_{\varepsilon,A,V}$ to $\mathcal{N}_{\varepsilon
,A,V}^{\tau}$ are precisely the nontrivial solutions to (\ref{prob}) which
satisfy (\ref{tau-inv}).

Since $\mathbb{S}^{1}$ acts on $H_{\varepsilon,A}^{1}(\mathbb{R}^{3},\mathbb{C}%
)^{\tau}$ by scalar multiplication: $(e^{\mathrm{i}\theta},u)\mapsto
e^{\mathrm{i}\theta}u$, the Nehari manifold $\mathcal{N}_{\varepsilon
,A,V}^{\tau}$ and the functional $J_{\varepsilon,A,V}$ are invariant under
this action. Therefore, if $u$ is a critical point of
$J_{\varepsilon,A,V}$ on $\mathcal{N}_{\varepsilon,A,V}^{\tau}$ then so is $\gamma u$
for every $\gamma\in\mathbb{S}^{1}.$ The set $\mathbb{S}^{1}u=\{\gamma
u\mid \gamma\in\mathbb{S}^{1}\}$ is then called a $\tau$\emph{-intertwining
critical }$\mathbb{S}^{1}$\emph{-orbit of }$J_{\varepsilon,A,V}.$ Two solutions of
(\ref{prob}) are said to be \emph{geometrically different} if their
$\mathbb{S}^{1}$-orbits are different.

Recall that $J_{\varepsilon,A,V} \colon \mathcal{N}_{\varepsilon,A,V}^{\tau
}\rightarrow\mathbb{R}$ is said to satisfy the \emph{Palais-Smale condition}
$(PS)_{c}$ at the level $c$ if every sequence $(u_{n})$ such that
\begin{equation*} 
u_{n} \in\mathcal{N}_{\varepsilon,A,V}^{\tau}, \quad 
J_{\varepsilon
,A,V}(u_{n}) \rightarrow c , \quad 
\nabla_{\mathcal{N}_{\varepsilon
,A,V}^{\tau}}J_{\varepsilon,A,V}(u_{n}) \rightarrow0 
        \end{equation*}
contains a convergent subsequence. Here $\nabla_{\mathcal{N}_{\varepsilon
,A,V}^{\tau}}J_{\varepsilon,A,V}(u)$ denotes the orthogonal projection of
$\nabla_{\varepsilon}J_{\varepsilon,A,V}(u)$ onto the tangent space to
$\mathcal{N}_{\varepsilon,A,V}^{\tau}$ at $u$.

In Lemma 3.4 of~\cite{ccs2} the following result was proved for $\varepsilon =1$.

\begin{proposition}
\label{palaissmale} For every $\varepsilon>0$, the functional $J_{\varepsilon
,A,V}\colon \mathcal{N}_{\varepsilon,A,V}^{\tau}\rightarrow\mathbb{R}$ satisfies
(PS)$_{c}$ at each level
\[
c<\varepsilon^{3}\min_{x\in\mathbb{R}^{3}\setminus\{0\}}(\#Gx)V_{\infty}%
^{3/2}E_{1},
\]
where $V_{\infty}=\liminf_{\left\vert x\right\vert \rightarrow\infty}V(x).$
\end{proposition}

\section{The limit problem}\label{embedding}

For any positive real number $\lambda$ we consider the
problem
\begin{equation}%
\begin{cases}
-\Delta u+\lambda u=(\frac{1}{\left\vert x\right\vert }\ast u^{2})u,\\
u\in H^{1}(\mathbb{R}^{3},\mathbb{R}).
\end{cases}
\label{limlambda}%
\end{equation}
Its associated energy functional $J_{\lambda}\colon H^{1}(\mathbb{R}%
^{3},\mathbb{R})\rightarrow\mathbb{R}$ is given by%
\[
J_{\lambda}(u)=\frac{1}{2}\left\Vert u\right\Vert _{\lambda}^{2}-\frac{1}%
{4}\mathbb{D}(u),\text{ \ \ \ with \ }\left\Vert u\right\Vert _{\lambda}%
^{2}=\int_{\mathbb{R}^{3}}\left(  \left\vert \nabla u\right\vert ^{2}+\lambda
u^{2}\right)  .
\]
Its Nehari manifold will be denoted by
\[
\mathcal{M}{_{\lambda}}=\left\{  u\in H^{1}(\mathbb{R}^{3},\mathbb{R}%
)\mid \,u\neq0,\quad\left\Vert u\right\Vert _{\lambda}^{2}=\mathbb{D}(u)\right\}
.
\]
We set%
\[
E_{\lambda}=\inf_{u\in\mathcal{M}{_{\lambda}}}J_{\lambda}(u).
\]
The critical points of $J_{\lambda}$ on $\mathcal{M}{_{\lambda}}$ are the
nontrivial solutions to (\ref{limlambda}). Note that $u$ solves
the real-valued problem%
\begin{equation}
\left\{
\begin{array}
[c]{c}%
-\Delta u+u= (\frac{1}{|x|} \ast u^2) u,\\
u\in H^{1}(\mathbb{R}^{3},\mathbb{R})
\end{array}
\right.  \label{lim}%
\end{equation}
if and only if $u_{\lambda}(x)=\lambda u(\sqrt{\lambda}x)$ solves
(\ref{limlambda}). Therefore,%
\[
E_{\lambda}=\lambda^{3/2}E_{1}.
\]
where $E_1$ is the least energy of a nontrivial solution to $(\ref{lim})$.
Minimizers of $J_{\lambda}$ on $\mathcal{M}{_{\lambda}}$ are called ground
states. The existence and uniqueness of ground
states up to sign and translations was established by
Lieb in \cite{lieb}. We denote by $\omega_{\lambda}$ the
positive solution to problem (\ref{limlambda}) which is radially symmetric
with respect to the origin.

Fix a radial function $\varrho\in C^{\infty}(\mathbb{R}^{3},\mathbb{R})$ such
that $\varrho(x)=1$ if $\left\vert x\right\vert \leq\frac{1}{2}$ and
$\varrho(x)=0$ if $\left\vert x\right\vert \geq1.$ For $\varepsilon>0$ set
$\varrho_{\varepsilon}(x)=\varrho(\sqrt{\varepsilon}x)$, $\omega
_{\lambda,\varepsilon}=\varrho_{\varepsilon}\omega_{\lambda}$ $\ $and
\begin{equation}
\upsilon_{\lambda,\varepsilon}=\frac{\left\Vert \omega_{\lambda,\varepsilon
}\right\Vert _{\lambda}}{\sqrt{\mathbb{D}(\omega_{\lambda,\varepsilon})}
}\,\omega_{\lambda,\varepsilon}. \label{bumps}%
\end{equation}
Note that $\operatorname{supp}(\upsilon_{\lambda,\varepsilon})\subset
B(0,1/\sqrt{\varepsilon})=\{x\in\mathbb{R}^{3}\mid \left\vert x\right\vert
\leq1/\sqrt{\varepsilon}\}$ and $\upsilon_{\lambda,\varepsilon}\in
\mathcal{M}_{\lambda}$. An easy computation shows that%
\begin{equation}
\lim_{\varepsilon\rightarrow0}J_{\lambda}(\upsilon_{\lambda,\varepsilon
})=\lambda^{3/2}E_{1}. \label{enerbumps}%
\end{equation}

Now we define
\[
\ell_{G, V}=\inf_{x\in\mathbb{R}^{N}}(\#Gx) {V^{3/2}(x)}
\]
and consider the set%
\[
M_{\tau}=\{x\in\mathbb{R}^{N}\mid (\#Gx) V^{3/2}(x)=\ell
_{G,V},\text{ }G_{x}\subset\ker\tau\}.
\]
Here $Gx=\{gx\mid g\in G\}$ is the $G$-orbit of the point $x\in\mathbb{R}^{3},$
$\#Gx$ is its cardinality, and $G_{x}=\{g\in G\mid gx=x\}$ is its isotropy
subgroup.  Observe that the points
in $M_{\tau}$ are not necessarily local minima of $V$.

In what follows we will assume  that there exists $\alpha>0$ such that the set%
\[
\left\{  y\in\mathbb{R}^{3} \mid (\#Gy)V^{3/2}(y)\leq\ell_{G,V}+\alpha\right\}
\]
is compact. Then%
\[
M_{G,V}=\left\{  y\in\mathbb{R}^{3}\mid (\#Gy)V^{3/2}(y)=\ell_{G,V}\right\}
\]
is a compact $G$-invariant set and all $G$-orbits in $M_{G,V}$ are finite. We
split $M_{G,V}$ according to the orbit type of its elements,
choosing subgroups $G_{1},\ldots,$ $G_{m}$ of $G$ such that the isotropy subgroup
$G_{x}$ of every point $x\in M_{G,V}$ is conjugate to precisely one of the
$G_{i}$'s, and we set%
\[
M_{i}=\left\{  y\in M_{G,V}\mid G_{y}=gG_{i}g^{-1}\text{ for some }g\in
G\right\}  .
\]
Since isotropy subgroups satisfy $G_{gx}=gG_{x}g^{-1},$ the sets $M_{i}$ are
$G$-invariant and, since $V$ is continuous, they are closed and pairwise
disjoint, and
\[
M_{G,V}=M_{1}\cup\cdots\cup M_{m}.
\]
Moreover, since
\[
\left\vert G/G_{i}\right\vert V^{3/2}(y)=(\#Gy)V^{3/2}(y)=\ell_{G,V}\text{
\ \ \ for all \ }y\in M_{i},
\]
the potential $V$ is constant on each $M_{i}.$ Here $\left\vert G/G_{i}%
\right\vert $ denotes the index of $G_{i}$ in $G.$ We denote by $V_{i}$ the
value of $V$ on $M_{i}.$

It is well known that the map $G/G_{\xi}\rightarrow G\xi$ given by $gG_{\xi
}\mapsto g\xi$ is a homeomorphism, see e.g. \cite{tD}. So, if $G_{i}%
\subset\ker\tau$ and $\xi\in M_{i},$ then the map
\[
G\xi\rightarrow\mathbb{S}^{1},\ \ \ \ g\xi\mapsto\tau(g),
\]
is well defined and continuous.

Let $\upsilon_{i,\varepsilon}=\upsilon_{V_{i},\varepsilon}$ be defined as in
(\ref{bumps}) with $\lambda=V_{i}.$ Set
\begin{equation}
\psi_{\varepsilon,\xi}(x)=\sum_{g\xi\in G\xi}\tau(g)\upsilon_{i,\varepsilon
} \Big (  \frac{x-g\xi}{\varepsilon} \Big )  e^{-\mathrm{i}A(g\xi)\cdot\left(
\frac{x-g\xi}{\varepsilon}\right)  }. \label{psi}%
\end{equation}
Let  $\pi_{\varepsilon,A,V}\colon H_{\varepsilon,A}^{1}
(\mathbb{R}^{3},\mathbb{C})^{\tau}\setminus\{0\}\rightarrow\mathcal{N}%
_{\varepsilon,A,V}^{\tau}$ be the radial projection given by%
\begin{equation}
\pi_{\varepsilon,A,V}(u)=\frac{\varepsilon\left\Vert u\right\Vert
_{\varepsilon,A,V}}{\sqrt{\mathbb{D}(u)}}u. \label{radproj}%
\end{equation}
We can derive the following results, arguing as  in  Lemmas  2
in \cite{cc1} (see also Lemma 4.2 in \cite{ccs}).

\begin{lemma}
\label{lemin}Assume that $G_{i}\subset\ker\tau$. Then, the following
hold:
\begin{itemize}
\item[(a)] For every $\xi\in M_{i}$ and $\varepsilon>0,$ one has that%
\[
\psi_{\varepsilon,\xi}(gx)=\tau(g)\psi_{\varepsilon,\xi}(x)\quad \forall g\in G,\text{ }x\in\mathbb{R}^{3}.
\]
\item[(b)] For every $\xi\in M_{i}$ and $\varepsilon>0,$ one has that%
\[
\tau(g)\psi_{\varepsilon,g\xi}(x)=\psi_{\varepsilon,\xi}(x)\quad \forall g\in G,\text{ }x\in\mathbb{R}^{3}.
\]
\item[(c)] One has that
\[
\lim_{\varepsilon\rightarrow0}{\varepsilon}^{-3}J_{\varepsilon,A,V}\left[
\pi_{\varepsilon,A,V}(\psi_{\varepsilon,\xi})\right]  =\ell_{G,V}E_{1}.
\]
uniformly in $\xi\in M_{i}.$
\end{itemize}
\end{lemma}

Let
\[
M_{\tau}=\left\{  y\in M_{G,V}\mid G_{y}\subset\ker\tau\right\}  =\bigcup
_{G_{i}\subset\ker\tau}M_{i}.
\]
As immediate consequence of  Lemma \ref{lemin}, we derive the following result.

\begin{proposition}
\label{ingoing}The map $\widehat{\iota}_{\varepsilon}\colon M_{\tau}\rightarrow
\mathcal{N}_{\varepsilon,A,V}^{\tau}$ given by%
\[
\widehat{\iota}_{\varepsilon}(\xi)=\pi_{\varepsilon,A,V}(\psi_{\varepsilon
,\xi})
\]
is well defined and continuous, and satisfies%
\[
\tau(g)\widehat{\iota}_{\varepsilon}(g\xi)=\widehat{\iota}_{\varepsilon}%
(\xi)\text{ \ \ \ \ }\forall\xi\in M_{\tau},\text{ }g\in G.
\]
Moreover, given $d>\ell_{G,V}E_{1},$ there exists $\varepsilon_{d}>0$ such that%
\[
\varepsilon^{-3}\ J_{\varepsilon,A,V}(\widehat{\iota}_{\varepsilon}(\xi))\leq
d\text{ \ \ \ \ }\forall\xi\in M_{\tau},\text{ }\varepsilon\in(0,\varepsilon
_{d}).
\]

\end{proposition}

\section{The baryorbit map}

\label{baryorbit}

Let us consider the real-valued problem%
\begin{equation} \label{real}%
\begin{cases}
-\varepsilon^{2}\Delta v+V(x)v=\frac{1}{\varepsilon^{2}}\left(  \frac
{1}{\left\vert x\right\vert }\ast u^{2}\right)  u,\\
v\in H^{1}(\mathbb{R}^{3},\mathbb{R}),\\
v(gx)=v(x)\text{ \ }\forall x\in\mathbb{R}^{3},\text{ }g\in G.
\end{cases}
\end{equation}
Set
\[
H^{1}(\mathbb{R}^{3},\mathbb{R})^{G}=\{v\in H^{1}(\mathbb{R}^{3}%
,\mathbb{R})\mid v(gx)=v(x)\text{ }\forall x\in\mathbb{R}^{3},\text{ }g\in G\}
\]
and write%
\[
\left\Vert
v\right\Vert_{V}^{2}=\int_{\mathbb{R}^{3}}\left(  \varepsilon^2\left\vert
\nabla v\right\vert ^{2}+V(x)v^{2}\right) .
\]
The nontrivial solutions of (\ref{real}) are the critical points of the energy
functional
\[
J_{\varepsilon,V}(v)=\frac{1}{2}\left\Vert v\right\Vert _{\varepsilon,V}%
^{2}-\frac{1}{4\varepsilon^{2}}\mathbb{D}(v)
\]
on the Nehari manifold%
\[
\mathcal{M}_{\varepsilon,V}^{G}=\left\{v\in H^{1}(\mathbb{R}^{3},\mathbb{R}%
)^{G}\mid v\neq0,\text{ }\left\Vert v\right\Vert _{\varepsilon,V}^{2}%
={\varepsilon^{-2}}\mathbb{D}(v)\right\}.
\]
Set%
\begin{equation}
c_{\varepsilon,V}^{G}=\inf_{\mathcal{M}_{\varepsilon,V}^{G}}J_{\varepsilon
,V}=\inf_{\substack{v\in H^{1}(\mathbb{R}^{3},\mathbb{R})^{G}\\v\neq0}%
}\frac{\varepsilon^{2}\left\Vert v\right\Vert _{\varepsilon,V}^{4}%
}{4\mathbb{D}(v)}. \label{infepsV}%
\end{equation}
As proved in Lemma 5.1 in \cite{ccs} we have
\begin{lemma}
\label{cotas}
There results~$0<(\inf_{\mathbb{R}^{3}}V)^{3/2}E_{1}\leq\varepsilon
^{-3}c_{\varepsilon,V}^{G}$ \ for every $\varepsilon>0,$ and%
\[
\limsup_{\varepsilon\rightarrow0}\varepsilon^{-3}c_{\varepsilon,V}^{G}\leq
\ell_{G,V}E_{1},
\]
\end{lemma}
We fix
$\hat{\rho}>0$ such that%
\begin{equation}%
\begin{cases}
\left\vert y-gy\right\vert >2\hat{\rho} & \text{if $gy\neq y\in
M_{G,W}$},\\
\text{dist}(M_{i},M_{j})>2\hat{\rho} & \text{if $i\neq j$},
\end{cases}
\label{robar}%
\end{equation}
where $G_{i}$, $M_{i}$, $V_i$ are the groups, the
sets and the values defined as in Section \ref{embedding}.

For $\rho\in(0,\hat{\rho}),$ let%
\[
M_{i}^{\rho}=\{y\in\mathbb{R}^{3}\colon\operatorname{dist}(y,M_{i})\leq\rho,\text{
\ }G_{y}=gG_{i}g^{-1}\text{ for some }g\in G\},
\]
and for each $\xi\in M_{i}^{\rho}$ and $\varepsilon>0,$ define%
\[
\theta_{\varepsilon,\xi}(x)=%
\sum_{g\xi\in G\xi}
\omega_{i} \Big (  \frac{x-g\xi}{\varepsilon} \Big )  ,
\]
where $\omega_{i}$ is unique positive ground state of problem (\ref{limlambda}%
) with $\lambda=V_{i}$ which is radially symmetric with respect to the
origin. Set%
\[
\Theta_{\rho,\varepsilon}=\left\{\theta_{\varepsilon,\xi}\mid \xi\in M_{1}^{\rho}%
\cup\cdots \cup M_{m}^{\rho}\right\}.
\]
Arguing as in Proposition 5 in \cite{ccs2}, we can derive the following result.

\begin{proposition}
\label{unique}Given $\rho\in(0,\hat{\rho})$ there exist $d_{\rho}%
>\ell_{G,V}E_{1}$ and $\varepsilon_{\rho}>0$ with the following property: For
every $\varepsilon\in(0,\varepsilon_{\rho})$ and every $v\in\mathcal{M}%
_{\varepsilon,V}^{G}$ with $J_{\varepsilon,V}(v)\leq\varepsilon^{3}d_{\rho}$
there exists precisely one $G$-orbit $G\xi_{\varepsilon,v}$ with
$\xi_{\varepsilon,v}\in M_{1}^{\rho}\cup\cdot\cdot\cdot\cup M_{m}^{\rho}$ such
that%
\[
\varepsilon^{-3}\left\Vert \left\vert v\right\vert -\theta_{\varepsilon
,\xi_{\varepsilon,v}}\right\Vert _{\varepsilon,V}^{2}=\min_{\theta\in
\Theta_{\rho,\varepsilon}}\left\Vert \left\vert v\right\vert -\theta
\right\Vert _{\varepsilon,V}^{2}.
\]
\end{proposition}

For every $c\in\mathbb{R}$ we set
\[
J_{\varepsilon,V}^{c}=\left\{v\in\mathcal{M}_{\varepsilon}^{G}\mid J_{\varepsilon,V
}(v)\leq c \right\}.
\]
Proposition \ref{unique} allows us to define, for each $\rho\in(0,\widehat{\rho})$
and $\varepsilon\in(0,\varepsilon_{\rho}),$ a local baryorbit map
\[
\widehat{\beta}_{\rho,\varepsilon,0}\colon J_{\varepsilon,V}^{\varepsilon^{3}d_{\rho}%
}\longrightarrow\left(  M_{1}^{\rho}\cup\cdot\cdot\cdot\cup M_{m}^{\rho
}\right)  /G
\]
by taking
\[
\widehat{\beta}_{\rho,\varepsilon,0}(v)=G\xi_{\varepsilon,v},
\]
where $G\xi_{\varepsilon,v}$ is the unique $G$-orbit given by the previous proposition.

Coming back to our original problem, for every $c\in\mathbb{R}$ set%
\[
J_{\varepsilon,A,V}^{c}=\{u\in\mathcal{N}_{\varepsilon,A,V}^{\tau}\mid J_{\varepsilon,A,V
}(u)\leq c\}.
\]
The following holds.

\begin{corollary}
\label{bary}For each $\rho\in(0,\widehat{\rho})$ and $\varepsilon\in
(0,\varepsilon_{\rho}),$ the local baryorbit map
\[
\widehat{\beta}_{\rho,\varepsilon}\colon J_{\varepsilon,A,V}^{\varepsilon^{3}d_{\rho}%
}\longrightarrow\left(  M_{1}^{\rho}\cup\cdot\cdot\cdot\cup M_{m}^{\rho
}\right)  /G,
\]
given by%
\[
\widehat{\beta}_{\rho,\varepsilon}(u)=\widehat{\beta}_{\rho,\varepsilon
,0}(\hat{\pi}_{\varepsilon}(\left\vert u\right\vert )),
\]
where $\hat{\pi}_{\varepsilon}\colon H^{1}(\mathbb{R}^{3},\mathbb{R})^{G}
\setminus\{0\}\rightarrow\mathcal{M}_{\varepsilon}^{G}$ is the radial
projection, is well defined and continuous. It satisfies%
\begin{align*}
\widehat{\beta}_{\rho,\varepsilon}(\gamma u)  &  =\widehat{\beta}%
_{\rho,\varepsilon}(u)\text{ \ \ \ \ }\forall\gamma\in\mathbb{S}^{1},\\
\widehat{\beta}_{\rho,\varepsilon}(\widehat{\iota}_{\varepsilon}(\xi))  &
=\xi\text{ \ \ \ \ }\forall\xi\in M_{\tau}\text{ with }J_{\varepsilon,A,V}%
(\iota_{\varepsilon}(\xi))\leq\varepsilon^{3}d_{\rho},
\end{align*}
where $\widehat{\iota}_{\varepsilon}$ is the map defined in \emph{Proposition
\ref{ingoing}}.
\end{corollary}
\begin{proof}
If $u\in\mathcal{N}_{\varepsilon,A,V}^{\tau}$ then $\hat{\pi}_{\varepsilon
}(\left\vert u\right\vert )\in\mathcal{M}_{\varepsilon}^{G}.$ The diamagnetic
inequality yields
\begin{equation}
J_{\varepsilon,V}(\hat{\pi}_{\varepsilon}(\left\vert u\right\vert ))\leq
J_{\varepsilon,A,V}(u). \label{inf}%
\end{equation}
So if $J_{\varepsilon,A,V}(u)\leq\varepsilon^{3}d_{\rho}$ then $\widehat{\beta
}_{\rho,\varepsilon}(u)$ is well defined. It is straightforward to verify that
it has the desired properties.
\end{proof}

\begin{corollary}
\label{infepsilon}If there exists $\xi\in\mathbb{R}^{3}$ such that
$(\#G\xi)V^{3/2}(\xi)=\ell_{G,V}$ and $G_{\xi}\subset\ker\tau$, then%
\[
\lim_{\varepsilon\rightarrow\infty}\varepsilon^{-3}c_{\varepsilon,A,V}^{\tau
}=\ell_{G,V}E_{1},
\]
where $c_{\varepsilon,A,V}^{\tau}=\inf_{{\mathcal{N}_{\varepsilon
,A,V}^{\tau}}
}J_{\varepsilon,A,V}.$
\end{corollary}

\begin{proof}
Inequality (\ref{inf}) yields $c_{\varepsilon,V}^{G}=\inf_{\mathcal{M}%
_{\varepsilon,V}^{G}}J_{\varepsilon,V}\leq\inf_{\mathcal{N}_{\varepsilon
,A,V}^{\tau}}J_{\varepsilon,A,V}=c_{\varepsilon,A,V}^{\tau}.$ 
        By statement (c) of Lemma \ref{lemin},
\[
\ell_{G,V}E_{1}=\lim_{\varepsilon\rightarrow\infty}\varepsilon^{-3}%
c_{\varepsilon,V}^{G}\leq\liminf_{\varepsilon\rightarrow0}\varepsilon
^{-3}c_{\varepsilon,A,V}^{\tau}\leq\limsup_{\varepsilon\rightarrow\infty
}\varepsilon^{-3}c_{\varepsilon,A,V}^{\tau}\leq\ell_{G,V}E_{1}. \qedhere
\]
\end{proof}

\section{Multiplicity results via Equivariant Morse theory}

We start by reviewing some well known facts on equivariant
Morse theory. We refer the reader to \cite{chang,wa} for further details.

\begin{definition}
Let $\Gamma$ be a compact Lie group and $X$ be a $\Gamma$-space.
\begin{itemize}
\item The $\Gamma
$-orbit of a point $x\in X$ is the set $\Gamma x:=\{\gamma x \mid \gamma\in \Gamma\}$.
\item A subset $A$ of $X$ is said to be $\Gamma$-invariant if $\Gamma
x\subset A$ for every $x\in A$. The $\Gamma$-orbit space of $A$ is the set
$A/\Gamma:=\{\Gamma x:x\in A\}$ with the quotient space topology.
\item $X$ is
called a free $\Gamma$-space if $\gamma x\neq x$ for every $\gamma\in\Gamma$,
$x\in X$.
\item  A map $f\colon X\rightarrow Y$ between $\Gamma$-spaces is called $\Gamma
$-invariant if $f$ is constant on each $\Gamma$-orbit of $X$, and it is called
$\Gamma$-equivariant if $f(\gamma x)=\gamma f(x)$ for every $\gamma\in\Gamma$, $x\in X$.
\end{itemize}
\end{definition}
 
We fix a field ${\mathbb{K}}$ and denote by $\mathcal{H}^{\ast}(X,A)$ the
Alexander-Spanier cohomology of the pair $(X,A)$ with coefficients in
${\mathbb{K}}$. If $X$ is a $\Gamma$-pair, i.e. if $X$ is a $\Gamma$-space and
$A$ is a $\Gamma$-invariant subset of $X,$ we write
\[
\mathcal{H}_{\Gamma}^{\ast}(X,A):=\mathcal{H}^{\ast}(E\Gamma\times_{\Gamma
}X,E\Gamma\times_{\Gamma}A)
\]
for the Borel-cohomology that pair. $E\Gamma$ is the total space of the
classifying $\Gamma$-bundle and $E\Gamma\times_{\Gamma}X$ is the orbit space
$\left(  E\Gamma\times X\right)  /\Gamma$ (see e.g. \cite[Chapter III]{tD}).
If $X$ is a free $\Gamma$-space, as will be the case in our application, then
the projection $E\Gamma\times_{\Gamma}X\rightarrow X/\Gamma$ is a homotopy
equivalence and it induces an isomorphism
\begin{equation} \label{eq:5.21}
\mathcal{H}_{\Gamma}^{\ast}(X,A)\cong\mathcal{H}^{\ast}(X/\Gamma,A/\Gamma).
\end{equation}
  In our setting, $\Gamma = \mathbb{S}^1$; if $A\subset X$ are $\mathbb{S}^{1}$-invariant subsets of
$\mathcal{N}_{\varepsilon,A,V}^{\tau}$ we denote by $X/\mathbb{S}^{1}$ and
$A/\mathbb{S}^{1}$ their $\mathbb{S}^{1}$-orbit spaces and by (\ref{eq:5.21}) it is legitimate to write
\[
\mathcal{H}_{\mathbb{S}^{1}}^{\ast}(X,A) \simeq \mathcal{H}^{\ast}(X/\mathbb{S}%
^{1},A/\mathbb{S}^{1}).
\]
If $\mathbb{S}^{1}u$ is an isolated critical $\mathbb{S}^{1}$-orbit of
$J_{\varepsilon,A,V}$ its $k$\emph{-th critical group} is defined as%
\[
C_{\mathbb{S}^{1}}^{k}(J_{\varepsilon,A,V},\mathbb{S}^{1}u)=\mathcal{H}%
_{\mathbb{S}^{1}}^{k}(J_{\varepsilon,A,V}^{c}\cap U,(J_{\varepsilon,A,V}%
^{c}\smallsetminus\mathbb{S}^{1}u)\cap U),
\]
where $U$ is an $\mathbb{S}^{1}$-invariant neighborhood of $\mathbb{S}^{1}u$
in $\mathcal{N}_{\varepsilon,A,V}^{\tau},$ $c=J_{\varepsilon,A,V}(u)$.
Its total dimension%
\[
\mu(J_{\varepsilon},\mathbb{S}^{1}u)=
{\sum\limits_{k=0}^{\infty}} 
\dim C_{\mathbb{S}^1}^{k}(J_{\varepsilon,A,V},\mathbb{S}^{1}u)
\]
is called the \emph{multiplicity of }$\mathbb{S}^{1}u$. If $\mathbb{S}^{1}u$
is nondegenerate and $J_{\varepsilon,A,V}$ satisfies the Palais-Smale condition in
some neighborhood of $c,$ then 
        $$
        \dim C_{\mathbb{S}^1}^{k}(J_{\varepsilon,A,V} 
,\mathbb{S}^{1}u)=1
        $$
         if $k$ is the Morse index of $J_{\varepsilon,A,V}$ at the
critical submanifold $\mathbb{S}^{1}u$ of $\mathcal{N}_{\varepsilon,A,V}^{\tau}$
and it is $0$ otherwise.

Moreover, for
$\rho>0$ we set 
\[
B_{\rho}M_{\tau}=\{x\in\mathbb{R}^{3}\mid \text{dist}(x,M_{\tau})\leq\rho\}
\]
and write $i_{\rho}\colon M_{\tau}/G\hookrightarrow B_{\rho}M_{\tau}/G$ for the
embedding of the $G$-orbit space of $M_{\tau}$ in that of $B_{\rho}M_{\tau}$.
We will show that this embedding has an effect on the number of solutions of
(\ref{prob}) for $\varepsilon$ small enough.

\begin{lemma}
\label{dim}For every $\rho\in(0,\widehat{\rho})$ and $d\in(\ell_{G,V}
         E_1,d_{\rho}),$ with $d_{\rho}$ as in \emph{Proposition \ref{unique}}, there
exists $\varepsilon_{\rho,d}>0$ such that%
\[
\dim\mathcal{H}^{k}(J_{\varepsilon,A,V}^{\varepsilon^{3}d}/\mathbb{S}^{1}%
)\geq\text{\emph{rank}}\left(  i_{\rho}^{\ast}\colon \mathcal{H}^{k}(B_{\rho}%
M_{\tau}/G)\rightarrow\mathcal{H}^{k}(M_{\tau}/G)\right)
\]
for every $\varepsilon\in(0,\varepsilon_{\rho,d})$ and $k\geq0,$ where
$i_{\rho}\colon M_{\tau}/G\hookrightarrow B_{\rho}M_{\tau}/G$ is the inclusion map.
\end{lemma}
         
\begin{proof}[\bf Proof.]
Let $\varepsilon_{\rho,d}=\min\{\varepsilon_{d},\varepsilon_{\rho}\}$ where
$\varepsilon_{\rho}$ is as in Proposition \ref{unique} and $\varepsilon_{d}%
$\ is as in Proposition \ref{ingoing}.\ Fix $\varepsilon\in(0,\varepsilon
_{\rho,d}).$ Then,%
\[
J_{\varepsilon,A,V}(\widehat{\iota}_{\varepsilon}(\xi))\leq\varepsilon
^{3}d\text{ \ \ and \ }\widehat{\beta}_{\rho,\varepsilon}(\widehat{\iota
}_{\varepsilon}(\xi))=\xi\text{ \ \ \ \ \ }\forall\xi\in M_{\tau}.
\]
By Proposition \ref{ingoing} and Corollary \ref{bary} the maps%
\[
M_{\tau}/G\overset{\iota_{\varepsilon}}{\longrightarrow}J_{\varepsilon,A,V
}^{{\varepsilon}^{3} d}/\mathbb{S}^{1}\overset{\beta_{\rho,\varepsilon}%
}{\longrightarrow}B_{\rho} M/G
\]
given by $\iota_{\varepsilon}(G\xi)=\widehat{\iota}_{\varepsilon}(\xi)$ and
$\beta_{\rho,\varepsilon}(\mathbb{S}^{1}u)=\widehat{\beta}_{\rho,\varepsilon
}(u)$ are well defined and satisfy $\beta_{\rho,\varepsilon}(\iota
_{\varepsilon}(G\xi))=G\xi$ for all\ $\xi\in M_{\tau}$. Note that 
         $M_{\tau}= {\textstyle\bigcup} 
\{M_{i}\mid G_{i}\subset\ker\tau\}$ is the union of some connected components of
$M.$ Moreover, our choice of $\widehat{\rho}$ implies that $B_{\rho}M_{\tau}\cap
B_{\rho}\left(  M\smallsetminus M_{\tau}\right)  =\emptyset.$ Therefore the
inclusion $i_{\tau,\rho}\colon B_{\rho}M_{\tau}/G\hookrightarrow B_{\rho}M/G$
induces an epimorphism in cohomology. Since $\beta_{\rho,\varepsilon}%
\circ\iota_{\varepsilon}=i_{\tau,\rho}\circ i_{\rho}$ we conclude that
\begin{align*}
\dim\mathcal{H}^{k}(J_{\varepsilon,A,V}^{\varepsilon^{3}d}/\mathbb{S}^{1})  &
\geq\text{rank}(\iota_{\varepsilon}^{\ast}\colon \mathcal{H}^{k}(J_{\varepsilon
}^{\varepsilon^{3}d}/\mathbb{S}^{1})\rightarrow\mathcal{H}_{k}(M_{\tau}/G))\\
&  \geq\text{rank}((\beta_{\rho,\varepsilon}\circ\iota_{\varepsilon})^{\ast
}\colon\mathcal{H}^{k}(B_{\rho}M/G)\rightarrow\mathcal{H}_{k}(M_{\tau}/G))\\
&  =\text{rank}\left(  i_{\rho}^{\ast}\colon\mathcal{H}^{k}(B_{\rho}M_{\tau
}/G)\rightarrow\mathcal{H}^{k}(M_{\tau}/G)\right)  ,
\end{align*}
as claimed.
\end{proof}

We are ready to prove our main theorem.

\begin{theorem}
\label{mainthm}
Assume there exists $\alpha>0$ such that the set%
\begin{equation}
\{x\in\mathbb{R}^{3}\mid (\#Gx) V^{3/2}(x)\leq\ell_{G,V}%
+\alpha\}. \label{compactness}%
\end{equation}
is compact. Then, given $\rho>0$ and $\delta\in(0,\alpha),$ there exists
$\bar{\varepsilon}>0$ such that for every $\varepsilon \in (0, \bar \varepsilon)$ one of the following
two assertions holds:
\begin{itemize}
\item[(a)] $J_{\varepsilon,A,V}$ has a nonisolated $\tau
$-intertwining critical $\mathbb{S}^{1}$-orbit in the set  $J_{\varepsilon,A,V}%
^{-1}[\varepsilon^{3}(\ell_{G,V} E_1 -\delta),\varepsilon^{3}%
(\ell_{G,V} E_1  + \delta)]$.
\item[(b)] $J_{\varepsilon,A,V}$ has
finitely many $\tau$-intertwining critical $\mathbb{S}^{1}$-orbits
$\mathbb{S}^{1}u_{1}$, $\mathbb{S}^{1}u_{2}$, \ldots, $\mathbb{S}^{1}u_{m}$ in $J_{\varepsilon,A,V}%
^{-1}[\varepsilon^{3}(\ell_{G,V} E_1 -\delta),\varepsilon^{3}
(\ell_{G,V} E_1 +\delta)].$ They satisfy
\[
{\textstyle\sum\limits_{j=1}^{m}}
\dim C_{\mathbb{S}^{1}}^{k}(J_{\varepsilon,A,V},\mathbb{S}^{1}u_{j})\geq
\text{\emph{rank}}(i_{\rho}^{\ast}:\mathcal{H}^{k}(B_{\rho}M_{\tau
}/G)\rightarrow\mathcal{H}^{k}(M_{\tau}/G))
\]
for every $k\geq0.$
\end{itemize}
In particular, if every $\tau$-intertwining critical $\mathbb{S}^{1}$-orbit of
$J_{\varepsilon,A,V}$ in the set $J_{\varepsilon,A,V}^{-1}[\varepsilon^{3}(\ell_{G,V}
E_1 -\delta),\varepsilon^{3}(\ell_{G,V} E_1
+\delta)]$ is nondegenerate then, for every $k\geq0,$ there are at least
\[
\text{\emph{rank}}(i_{\rho}^{\ast}\colon \mathcal{H}^{k}(B_{\rho}M_{\tau
}/G)\rightarrow\mathcal{H}^{k}(M_{\tau}/G))
\]
of them having Morse index $k$ for every $k\geq0$.
\end{theorem}
\begin{proof}
Assume $M_{\tau}%
\neq\emptyset$ and let $\rho>0$ and $\delta\in(0,\alpha E_1 )$
be given. Without loss of generality we may assume that $\rho\in(0,\bar{\rho
})$. Assumption (\ref{compactness}) implies that
\[
\ell_{G,V}+\alpha\leq\min_{x\in\mathbb{R}^{3}\smallsetminus\{0\}}(\#Gx) V_{\infty}^{3/2}%
\]
where $V_{\infty}=\limsup_{|x|\rightarrow\infty}V(x)$. By Proposition \ref{palaissmale} the
functional 
\[
J_{\varepsilon,A,V}\colon \mathcal{N}_{\varepsilon,A,V}^{\tau}\rightarrow
\mathbb{R}
\] 
satisfies (PS)$_{c}$ at each level $c\leq  \varepsilon^{3}
(\ell_{G,V} E_1 +\delta)$ for every $\varepsilon>0.$ By Corollary \ref{infepsilon} there
exists $\varepsilon_{0}>0$ such that%
\[
\ell_{G,V} E_1 -\delta<\varepsilon^{-3}\inf_{u\in\mathcal{N}
_{\varepsilon}^{\tau}}J_{\varepsilon,A,V}\quad\forall\varepsilon
\in(0,\varepsilon_{0}).
\]
Let $d\in(\ell_{G,V} E_1 ,\min\{d_{\rho},\ell_{G,V} E_1
+\delta\})$ with $d_{\rho}$ as in Proposition \ref{unique}, and
$\overline{\varepsilon}=\min\{\varepsilon_{0},\varepsilon_{\rho,d}\}$ with
$\varepsilon_{\rho,d}$ as in Lemma \ref{dim}. Fix $\varepsilon\in
(0,\overline{\varepsilon})$ and for $u\in\mathcal{N}_{\varepsilon,A,V}^{\tau}$
with $J_{\varepsilon,A,V}(u)=c$ set%
\[
C_{\mathbb{S}^{1}}^{k}(J_{\varepsilon,A,V},\mathbb{S}^{1}u)=\mathcal{H}%
^{k}((J_{\varepsilon,A,V}^{c}\cap U)/\mathbb{S}^{1},((J_{\varepsilon,A,V}%
^{c}\smallsetminus\mathbb{S}^{1}u)\cap U)/\mathbb{S}^{1}).
\]
Assume that every critical $\mathbb{S}^{1}$-orbit of $J_{\varepsilon,A,V}$ lying in
$J_{\varepsilon,A,V}^{-1}[\varepsilon^{3}(\ell_{G,V} E_1 -\delta),$
$\varepsilon^{3}(\ell_{G,V} E_1 +\delta)]$ is isolated. Since
$J_{\varepsilon,A,V}\colon \mathcal{N}_{\varepsilon,A,V}^{\tau}\rightarrow\mathbb{R}$
satisfies (PS)$_{c}$ at each $c\leq\varepsilon^{3}
(\ell_{G,V} E_1 +\delta)$ there
are only finitely many of them. Let $\mathbb{S}^{1}u_{1},\ldots,\mathbb{S}%
^{1}u_{m}$ be those critical $\mathbb{S}^{1}$-orbits of $J_{\varepsilon,A,V}$ in
$\mathcal{N}_{\varepsilon,A,V}^{\tau}$ which satisfy $J_{\varepsilon,A,V}%
(u_{i})<\varepsilon^{3}d.$ Applying Theorem 7.6 in \cite{chang} to $J_{\varepsilon,A,V
}\colon \mathcal{N}_{\varepsilon,A,V}^{\tau}\rightarrow\mathbb{R}$ with $a=\varepsilon
^{3}(\ell_{G,V} E_1 -\delta)$ and $b=\varepsilon^{3}d$ and Lemma
\ref{dim} we obtain that%
\begin{align*}%
\sum_{j=1}^{m}
\dim C_{\mathbb{S}^{1}}^{k}(J_{\varepsilon,A,V},\mathbb{S}^{1}u_{i})  &  \geq
 \dim\mathcal{H}^{k}(J_{\varepsilon,A,V}^{\varepsilon^{3}d}/\mathbb{S}^{1})\\
&  \geq\text{rank}\left(  i_{\rho}^{\ast}\colon \mathcal{H}^{k}(B_{\rho}M_{\tau
}/G)\rightarrow\mathcal{H}^{k}(M_{\tau}/G)\right)
\end{align*}
for every $k\geq0$, as claimed. The last assertion of Theorem \ref{mainthm} is
an immediate consequence of Theorem 7.6 in \cite{chang}.
\end{proof}

If the inclusion $i_{\rho}:M_{\tau}/G\hookrightarrow B_{\rho}M_{\tau}/G$ is a
homotopy equivalence then
\[
\operatorname{rank}\left(  i_{\rho}^{\ast}\colon\mathcal{H}^{k}(B_{\rho}M_{\tau
}/G)\rightarrow\mathcal{H}^{k}(M_{\tau}/G)\right)  =\dim\mathcal{H}%
^{k}(M_{\tau}/G).
\]

An immediate consequence of Theorem \ref{mainthm} is the following.

\begin{corollary}
\label{maincor}If assumption \emph{(\ref{compactness})} holds then, given
$\rho>0$ and $\delta>0,$ there exists $\bar{\varepsilon}>0$ such that for
every $\varepsilon>0$ problem $(\ref{prob})$ has at least
\[ 
{\sum\limits_{k=0}^{\infty}} 
\operatorname{rank}(i_{\rho}^{\ast}:\mathcal{H}^{k}(B_{\rho}M_{\tau
}/G)\rightarrow\mathcal{H}^{k}(M_{\tau}/G))
\]
geometrically different solutions in
$J_{\varepsilon,A,V}^{-1}[\varepsilon^{3}(\ell_{G,V} E_1 -\delta),\varepsilon^{3}(\ell_{G,V}
E_1 +\delta)]$, counted with their multiplicity.
\end{corollary}

\subsection{Examples} \label{subs:5.1}

As a typical application of our existence result, we consider the constant 
         magnetic field $B(x_1,x_2,x_3) = (0,0,2)$ in $\mathbb{R}^3$. 
         We can consider its vector potential $A(x_1,x_2,x_3) = (-x_2,x_1,0)$,
          and identify $\mathbb{R}^3$ with $\mathbb{C} \times \mathbb{R}$. 
          With this in mind, we write $A(z,t) = (iz,0)$, with $z=x_1+ix_2$.
           We remark that $A(e^{i \theta}z,t) = e^{i\theta} A(z,t)$ 
           for every $\theta \in \mathbb{R}$.

Given $m \in \mathbb{N}$, $m \geq 1$ and $n \in \mathbb{Z}$, we look for 
           solutions $u$ to problem (\ref{prob}) which satisfy the symmetry property
\[
u \left( e^{2 \pi i k /m}z,t \right) = e^{2 \pi i n k /m} u \left( z,t \right)
\]
for every $k = 1, \ldots, m$ and $(z,t) \in \mathbb{C} \times \mathbb{R}$. We assume that $V$ satisfies
\begin{itemize}
\item[(a)] $V \in C^2(\mathbb{R}^3)$ is bounded and $\inf_{\mathbb{R}^3} V >0$; moreover
\[
\inf_{x \in \mathbb{R}^3} V^{3/2}(x) < \liminf_{|x| \to +\infty} V^{3/2}(x).
\]
\item[(b)] There exists $m_0 \in \mathbb{N}$ such that
\begin{align*}
m_0 \inf_{x \in \mathbb{R}^3} V^{3/2}(x) &< \inf_{t \in \mathbb{R}} V^{3/2}(0,t) \\
V \left( e^{2\pi i k/m_0} z,t \right) &= V(z,t)
\end{align*}
for every $k = 1,\ldots, m_0$ and $(z,t) \in \mathbb{C} \times \mathbb{R}$.
\end{itemize}
For each $m$ that divides $m_0$ (in symbols: $m|m_0$), we consider the group
\[
G_m = \left\{ e^{2\pi i k /m} \mid k = 1,\ldots, m \right\}
\]
acting by multiplication on the $z$-coordinate of each point 
           $(z,t) \in \mathbb{C} \times \mathbb{R}$. It is easy to check 
           that $A$ and $V$ match all the assumptions of 
           Theorem \ref{mainthm} for each $G=G_m$: the compactness 
           condition (\ref{compactness}) follows from the two 
           inequalities in (a) and (b). If $\tau \colon G_m \to \mathbb{S}^1$
            is any homeomorphism, we have that
\[
M_\tau = \left\{ x \in \mathbb{R}^3 \mid V(x) = \inf_{y \in \mathbb{R}^3} V(y) \right\}.
\]
Given $n \in \mathbb{Z}$, we consider the homeomorphism 
            $\tau \left( e^{2 \pi i k /m} \right) = e^{2\pi i nk/m}$.
             In particular, given $\rho$, $\delta>0$, for $\varepsilon$ 
             small enough we have
\[
\sum_{m|m_0} \sum_{k=0}^{\infty} m \operatorname{rank} 
             \left( i^*_\rho \colon \mathcal{H}^k (B_\rho M/G_m)
             \to \mathcal{H}^k (M/G_m) \right)
\]
geometrically distinct solutions, counted with multiplicity.
\begin{remark}
Our multiplicity result cannot be obtained, in general,
              via standard category arguments. For a concrete example,
               consider $M=\bigcup_{n \geq 1} S_n$, where
\[
S_n = \left\{ (x_1,x_2,x_3) \in \mathbb{R}^3 \mid \left( x_1 - 
               \frac{1}{n} \right)^2 + x_2^2 + x_3^2 = \frac{1}{n^2} \right\}.
\]
The category of $M$ is then $2$, whereas
\[
\lim_{\rho \to 0} \operatorname{rank} \left( i_\rho^* \colon 
               \mathcal{H}^{2} (B_\rho M) \to \mathcal{H}^{2}(M) \right) = +\infty.
\]
For a short proof, we refer to \cite[Example 1, pag. 1280]{cc2}
\end{remark}

\section{Appendix}

\begin{proposition}
The second derivative $J''_{\varepsilon,A,V}$ is continuous.
\end{proposition}
\begin{proof}
We first prove that $J''_{\varepsilon,A,V}$ is continuous at zero. Let $\{u_n\}_n$ be a sequence in $H^1_{\varepsilon,A}(\mathbb{R}^3,\mathbb{C})$ converging to zero. By Sobolev's embedding theorem, $u_n \to 0$ in $L^r(\mathbb{R}^3)$ for $r \in [2,6]$.
From (\ref{eq:23}) it follows that
\begin{multline} \label{eq:33}
\left| \int_{\mathbb{R}^3} \left( \int_{\mathbb{R}^3} \frac{|u_n(y)|^2}{|x-y|}dy \right) w(x) \overline{v(x)}\, dx \right| \\\leq C \Vert u_n \Vert^2_{L^{12/5}(\mathbb{R}^{3})}\Vert v \Vert_{L^{12/5}(\mathbb{R}^{3})} \Vert w \Vert_{L^{12/5}(\mathbb{R}^{3})}
\leq o(1) \| v\|_{\varepsilon, A, V} \ \| w\|_{\varepsilon, A, V}
\end{multline}
This implies that
\begin{equation}\label{flimit3}
\lim_{n \to +\infty}
\left|\operatorname{Re} \int_{\mathbb{R}^3} \left( \int_{\mathbb{R}^3} \frac{|u_n(y)|^2}{|x-y|}dy \right) w(x) \overline{v(x)}\, dx \right|=0
\end{equation}
whenever $u_n \to 0$ strongly in $H^1_{\varepsilon,A,V}(\mathbb{R}^3,\mathbb{C})$.

Similarly, we use (\ref{eq:23}) to prove that
\begin{multline*}
\left|\int_{\mathbb{R}^3} \left(  \frac{1}{|x|}* (u_n \overline{v})  \right)u_n \overline{w}
 \, dx \right|= \left| \int_{\mathbb{R}^3 \times \mathbb{R}^3} \frac{u_n(x) \overline{w(x)} u_n(y) \overline{v(y)}}{|x-y|}dx \, dy\right| \\
\leq C\left\Vert u_n \overline{w} \right\Vert_{L^{6/5}(\mathbb{R}^3)} \left\Vert u_n \overline{v} \right\Vert_{L^{6/5}(\mathbb{R}^3)} \\
\leq C \left\Vert u_n\right\Vert_{L^{12/5}(\mathbb{R}^3)}^2 \left\Vert v\right\Vert_{L^{12/5}(\mathbb{R}^3)} \left\Vert w\right\Vert_{L^{12/5}(\mathbb{R}^3)}
\end{multline*}
which implies that
\begin{equation}\label{flimit}
\lim_{n \to +\infty}\left| \operatorname{Re} \int_{\mathbb{R}^3}  \left(  \frac{1}{|x|}* ( u_n \overline{v})  \right) u_n \overline{w} \, dx
\right|=0
\end{equation}
whenever $u_n \to 0$ strongly in $H^1_{\varepsilon,A,V}(\mathbb{R}^3,\mathbb{C})$.
It is now easy to conclude that $J''_{\varepsilon,A,V}(u_n) \to J''_{\varepsilon,A,V}(0)$.

If $u_n \to u$ in $H^1_{\varepsilon,A}(\mathbb{R}^3,\mathbb{C})$, we replace $|u_n|^2$ in (\ref{eq:33}) with $u_n^0 = |u_n|^2-|u_n-u|^2-|u|^2$ and find
\begin{equation*}
\left| \int_{\mathbb{R}^3} \left( \int_{\mathbb{R}^3} \frac{|u_n^0(y)|}{|x-y|}dy \right) w(x) \overline{v(x)}\, dx \right| \leq C \|u_n^0\|_{L^{6/5}(\mathbb{R}^3)} |w\|_{\varepsilon,A,V} \|v\|_{\varepsilon,A,V}
\leq o(1) \|w\|_{\varepsilon,A,V} \|v\|_{\varepsilon,A,V}.
\end{equation*}
Analogously
\begin{multline*}
\left| \int_{\mathbb{R}^3} \left( \int_{\mathbb{R}^3} \frac{|u_n(y)-u(y)|^2}{|x-y|}dy \right) w(x) \overline{v(x)}\, dx \right| \\
\leq C \|u_n-u\|^2_{L^{12/5}(\mathbb{R}^3)} \|w\|_{\varepsilon,A,V} \|v\|_{\varepsilon,A,V}
\leq o(1) \|w\|_{\varepsilon,A,V} \|v\|_{\varepsilon,A,V},
\end{multline*}
we conclude that
\begin{equation} \label{eq:34}
\left|
\int_{\mathbb{R}^3} \left( \int_{\mathbb{R}^3} \frac{|u_n(y)|^2 - |u(y)|^2}{|x-y|}dy \right) w(x) \overline{v(x)} \, dx   \right|
\leq o(1) \|w\|_{\varepsilon,A,V} \|v\|_{\varepsilon,A,V}.
\end{equation}
Switching to the second term of $J''_{\varepsilon,A,V}(u_n)-J''_{\varepsilon,A,V}(u)$, we notice that
\begin{multline*}
\int_{\mathbb{R}^3 \times \mathbb{R}^3} \frac{u_n(x)\overline{w(x)}u_n(y)\overline{v(y)}}{|x-y|}dx \, dy - \int_{\mathbb{R}^3 \times \mathbb{R}^3} \frac{u(x)\overline{w(x)}u(y)\overline{v(y)}}{|x-y|}dx \, dy \\
=\int_{\mathbb{R}^3 \times \mathbb{R}^3} \frac{\left[ \left( u_n(y)-u(y) \right) u_n(x)+ \left( u_n(x)-u(x) \right) u(y) \right] \overline{v(y)} \overline{w(x)}}{|x-y|}dx \,dy,
\end{multline*}
so that
\begin{multline} \label{eq:35}
\left|
\int_{\mathbb{R}^3 \times \mathbb{R}^3} \frac{u_n(x)\overline{w(x)}u_n(y)\overline{v(y)}}{|x-y|}dx \, dy - \int_{\mathbb{R}^3 \times \mathbb{R}^3} \frac{u(x)\overline{w(x)}u(y)\overline{v(y)}}{|x-y|}dx \, dy \right| \\
\leq \left| \int_{\mathbb{R}^3 \times \mathbb{R}^3} \frac{\left( \left( u_n(y)-u(y) \right) u_n(x) \right) \overline{v(y)} \overline{w(x)}}{|x-y|}dx \,dy \right| \\
\quad {}+
\left| \int_{\mathbb{R}^3 \times \mathbb{R}^3} \frac{\left( \left( u_n(x)-u(x) \right) u(y) \right) \overline{v(y)} \overline{w(x)}}{|x-y|}dx \,dy \right| \\
\leq o(1) \left\Vert v \right\Vert_{\varepsilon,A,V} \left\Vert w \right\Vert_{\varepsilon,A,V}
\end{multline}
because $u_n \to u$. Recalling that $|\operatorname{Re}z| \leq |z|$ for every $z \in \mathbb{C}$ and putting together (\ref{eq:34}) and (\ref{eq:35}), we conclude that $J''_{\varepsilon,A,V}(u_n) \to J''_{\varepsilon,A,V}(u)$.
\end{proof}

\end{document}